\documentclass[12pt,reqno]{amsart}

\usepackage{multirow}
\usepackage{hhline}

\usepackage{times}
\usepackage[T1]{fontenc}
\usepackage{mathrsfs}
\usepackage{latexsym}
\usepackage[dvips]{graphics}
\usepackage[titletoc,title]{appendix}
\usepackage{epsfig}

\usepackage{amsmath,amsfonts,amsthm,amssymb,amscd}
\input amssym.def
\input amssym.tex
\usepackage{color}
\usepackage{hyperref}

\usepackage{color}
\usepackage{breakurl}
\usepackage{mleftright}
\usepackage{comment}
\newcommand{\bburl}[1]{\textcolor{blue}{\url{#1}}}

\newcommand{\be}{\begin{equation}}
\newcommand{\ee}{\end{equation}}
\newcommand{\bea}{\begin{eqnarray}}
\newcommand{\eea}{\end{eqnarray}}

\newtheorem{thm}{Theorem}[section]

\newtheorem{cor}[thm]{Corollary}
\newtheorem{lem}[thm]{Lemma}

\numberwithin{equation}{section}

\newcommand{\Mod}[1]{\ \mathrm{mod}\ #1}

\begin{document}

\title{On Arithmetic Progressions of Powers\\ in Cyclotomic Polynomials}

\author{H\`ung Vi\d{\^e}t Chu}
\email{\textcolor{blue}{\href{mailto:hungchu2@illinois.edu}{hungchu2@illinois.edu}}}
\address{Department of Mathematics, University of Illinois at Urbana-Champaign, Urbana, IL 61820, USA}

\subjclass[2010]{}


\thanks{The author wishes to thank the referees and the editor for useful comments that improved this paper.}

\date{\today}
\maketitle

\begin{abstract}
We determine necessary conditions for when powers corresponding to positive/negative coefficients of $\Phi_{n}$ are in arithmetic progression. When $n = pq$ for any primes $q>p>2$, our conditions are also sufficient. Finally, we generalize the result when $n = pq$ to the so-called inclusion-exclusion polynomials first introduced by Bachman. 
\end{abstract}

\section{Introduction and main results.}

For integers $n\ge 1$, the $n$th cyclotomic polynomial is defined as 
$$\Phi_n(X) \ =\ \prod_{m=1, (m,n)=1}^n (X-e^{\frac{2\pi mi}{n}}).$$
It is well known that $\Phi_n$ is in $\mathbb{Z}[X]$ with degree $\phi(n)$, where $\phi$ is the Euler totient function. In the study of cyclotomic polynomials, we can reduce our enquiry to the case when $n$ is odd, square-free, and composite by \cite[Remark 2.2]{Th}. Much work has been done to characterize $\Phi_n$ (see \cite{Ba, Bm, Dr, Th}), and many nice results are achieved when $n$ has a small number of prime divisors (see \cite{Be1, Be2, LL}).  In particular, we know an explicit formula for $\Phi_{pq}$:
\begin{align}\label{LLfor}\Phi_{pq}(X) \ =\ \left(\sum_{i=0}^r X^{ip}\right)\left(\sum_{j=0}^sX^{jq}\right) - \left(\sum_{i=r+1}^{q-1} X^{ip}\right)\left(\sum_{j=s+1}^{p-1}X^{jq}\right)X^{-pq},\end{align}
where $r, s$ are nonnegative and $pr+qs = (p-1)(q-1)$. For its derivation, see \cite{LL}. Clearly, $r$ and $s$, when $0<r<q$, are uniquely determined as follows:
\begin{align}
    \label{d1} pr &\ \equiv\ (p-1)(q-1) \Mod q\\
    \label{d2} s &\ =\ ((p-1)(q-1)-pr)/q.
\end{align}
If we expand the products in \eqref{LLfor}, the resulting monomial terms are all different \cite{LL}. 

Our first main result shows necessary conditions when powers of $X$  are in arithmetic progression. Two examples are
\begin{align*}
\Phi_{21}(X) &\ =\ X^{12}-X^{11}+X^9 - X^8 + X^6 - X^4+X^3 - X+1,\\
\Phi_{33}(X) &\ =\ X^{20} - X^{19} + X^{17} - X^{16} + X^{14} - X^{13} +\cdots - X^4 + X^3 - X +1.
\end{align*}
Observe that powers corresponding to positive coefficients of $\Phi_{21}(X)$ are in arithmetic progression, and powers corresponding to negative coefficients of $\Phi_{33}(X)$ are in arithmetic progression. Our theorems provide necessary conditions for when these arithmetic progressions appear. Let $c_{n,k}$ be the coefficient of $X^k$ and define $$S^+_{n} \ :=\ \{k\,:\, c_{n,k} > 0\}\mbox{ and }S^-_{n} \ :=\ \{k\,:\, c_{n,k} < 0\}.$$
\begin{thm}\label{S+}
Let $n$ be an odd, square-free, composite number. Write $n = p_1p_2\cdots p_t$, where $p_1<p_2<\cdots<p_t$. Then the following hold.
\begin{enumerate}
    \item[(i)] If $t$ is odd, then $S_n^+$ is not in arithmetic progression.
    \item[(ii)] If $t$ is even and $S_n^+$ is in arithmetic progression, then $p_2  \equiv 1\Mod p_1$.
\end{enumerate}
\end{thm}
\begin{thm}\label{S-}
Let $n$ be an odd, square-free, composite number. Write $n = p_1p_2\cdots p_t$, where $p_1<p_2<\cdots<p_t$. Then the following hold. 
\begin{enumerate}
    \item [(i)] If $t$ is odd, then $S_n^-$ is not in arithmetic progression.
    \item [(ii)] If $t$ is even and $S_n^-$ is in arithmetic progression, then $p_2 \equiv -1\Mod p_1$.
\end{enumerate}
\end{thm}
We have the following two corollaries. 
\begin{cor}
Let $n$ be an odd, square-free, composite number. Then $S_n^-$ and $S_n^+$ are not simultaneously in arithmetic progression. 
\end{cor}
\begin{proof} Write $n = p_1p_2\cdots p_t$, where $p_1<p_2<\cdots<p_t$. If $t$ is odd, Theorem \ref{S+} says that $S_n^+$ is not in arithmetic progression. Suppose that $t$ is even and that both $S^+_{n}$ and $S^-_{n}$ are in arithmetic progression. By items (ii) of Theorem \ref{S+} and Theorem \ref{S-}, $p_2 = m_1p_1+1 = m_2p_1-1$ for some $m_1, m_2\in\mathbb{N}$. Hence, $2p_2 = (m_1+m_2)p_1$, which implies that either $p_1 = 2$ or $p_1$ divides $p_2$. Both cases are impossible. 
\end{proof}
\begin{cor}\label{twoprimes}
Let $2<p<q$ be primes. Then $S^+_{pq}$ forms an arithmetic progression if and only if $q = mp+1$ for some $m\in\mathbb{N}$, and $S^-_{pq}$ forms an arithmetic progression if and only if $q = mp-1$ for some $m\in\mathbb{N}$. 
\end{cor}
\begin{proof}
Due to similarity, we only prove the result for $S^+_{pq}$. The forward implication follows directly from Theorem \ref{S+} item (ii). For the backward implication, we use formula \eqref{LLfor}. Suppose that $q = mp + 1$ for some $m\in\mathbb{N}$. Then formulas \eqref{d1} and \eqref{d2} give $r = m(p-1)$ and $s = 0$. Combined with formula \eqref{LLfor}, this clearly indicates that $S^+_{pq}$ is in arithmetic progression of difference $p$. 
\end{proof}

Finally, we generalize Corollary \ref{twoprimes} to a family of inclusion-exclusion polynomials introduced by Bachman \cite{Ba1}. An inclusion-exclusion polynomial is defined as 
$$P_{a,b}(X)\ =\ \frac{(X-1)(X^{ab}-1)}{(X^a-1)(X^b-1)},$$
where $a, b$ are relatively prime natural numbers; $P_{a,b}$ can also be interpreted as the semigroup polynomial of the numerical semigroup generated by $a$ and $b$ \cite{Mor}. When $a$ and $b$ are odd primes, $P_{a,b}(X) = \Phi_{ab}(X)$. 
\begin{thm}\label{i/e}
Let $1<a<b$ be coprime natural numbers. 
Then $P_{a,b}(X)$ is a polynomial. Furthermore, the exponents of the monomials with positive coefficient are in arithmetic progression if and only if $b\equiv 1\Mod a$. The exponents of the monomials with negative coefficient are in arithmetic progression if and only if $b\equiv -1\Mod a$. 
\end{thm}

\section{Proofs of Theorems \ref{S+} and \ref{S-}.}
We modify a powerful technique, which was used by Schur \cite{Leh} to prove there exist cyclotomic polynomials with coefficients arbitrarily large in absolute value. The following lemma is the key ingredient.
\begin{lem}\label{key}
Let $n$ be an odd, square-free, composite number. Write $n = p_1p_2\cdots p_t$. Then modulo $X^{p_2+2}$,
$$\Phi_n(X)\ \equiv\ \begin{cases}\sum_{i=0}^{p_1-1} X^i - X^{p_2} - X^{p_2+1}, &\mbox{ if } t\mbox{ is odd};\\
\sum_{i=0}^\infty X^{ip_1} - \sum_{i=0}^\infty X^{ip_1+1}+X^{p_2} - X^{p_2+1}, &\mbox{ if }t\mbox{ is even}.\end{cases}$$
\end{lem}
\begin{proof}
By \cite[Lemma 1.2]{Th}, we can write:
$\Phi_n(X) \ =\ \prod_{d|n}(X^{d}-1)^{\mu(n/d)}$,
where $\mu(n)$ denotes the Mobius function. Modulo $X^{p_2+2}$, we have
\begin{align*}\Phi_n(X) &\ \equiv\ (X-1)^{\mu(n)}(X^{p_1}-1)^{\mu(n/p_1)}(X^{p_2}-1)^{\mu(n/p_2)}\prod_{\substack{d|n\\ d>p_2}}(X^{d}-1)^{\mu(n/d)}\\
&\ \equiv\ (X-1)^{\mu(n)}(X^{p_1}-1)^{\mu(n/p_1)}(X^{p_2}-1)^{\mu(n/p_2)}(-1)^{\ell},
\end{align*}
where $\ell = \sum_{\substack{d|n\\ d>p_2}}\mu(n/d)$.

If $t$ is odd, $\mu(n) = -1$ and $\mu(n/p_1) = \mu(n/p_2) = 1$. Since it is well known that $\sum_{d|n} \mu(d) = 0$ for all $n>1$, we get $\ell = -1$. Thus, 
\begin{align*}\Phi_n(X) \ \equiv\ -\frac{(X^{p_1}-1)(X^{p_2}-1)}{X-1}&\ \equiv\ -(1+X+\cdots+X^{p_1-1})(X^{p_2}-1)\\
&\ \equiv\ 1 + X+ \cdots + X^{p_1-1} - X^{p_2} - X^{p_2+1}.
\end{align*}

If $t$ is even, we get $\mu(n) = 1$, $\mu(n/p_1) = \mu(n/p_2) = -1$ and $\ell = 1$. Therefore,
\begin{align*}
    \Phi_n(X) &\ \equiv\ -(X-1)\frac{1}{1-X^{p_1}}\frac{1}{1-X^{p_2}}\\
    &\ \equiv\ -(X-1)(1+X^{p_1}+X^{2p_1} + \cdots)(1+X^{p_2}+X^{2p_2} + \cdots)\\
    &\ \equiv\ \sum_{i=0}^\infty X^{ip_1} - \sum_{i=0}^\infty X^{ip_1+1}+X^{p_2} - X^{p_2+1}.
\end{align*}
We have finished our proof. 
\end{proof}
\begin{lem}\label{skey}
Let $n$ be an odd, square-free, composite number. Write $n = p_1p_2\cdots p_t$. Suppose $2p_1+2p_2\ge \phi(n)+2$, where $\phi$ is the Euler totient function. Then $t = 2$.
\end{lem}
\begin{proof}
Suppose that $t\ge 3$. Then \begin{align*}\phi(n) + 2\ \ge\ (p_1-1)(p_2-1)(p_3-1) + 2&\ \ge\ 6(p_1-1)(p_2-1)+ 2\\
&\ = \ 6p_1p_2 - 6(p_1+p_2) + 8.
\end{align*}
Hence, $2p_1+2p_2\ge \phi(n)+2$ implies that $4(p_1+p_2)\ge 3p_1p_2+4$, which is a contradiction since $p_1p_2 > 4p_1$ and $2p_1p_2>4p_2$. Therefore, $t = 2$. 
\end{proof}

We are ready to prove Theorems \ref{S+} and \ref{S-}.
\begin{proof}[Proof of Theorem \ref{S+}]
If $t$ is odd and $S_n^+$ is in arithmetic progression, then by Lemma \ref{key} it must be that $S_n^+ = \{0, 1, 2,\ldots, \phi(n)\}$. However, $X^{p_2}$ and $X^{p_2+1}$ have coefficient $-1$, a contradiction. Therefore, $S_n^+$ is not in arithmetic progression.  

If $t$ is even, Lemma \ref{key} guarantees that $0$ and $p_1$ are in $S_n^+$. Suppose that $S_n^+$ is in arithmetic progression. If $p_2\in S_n^+$, then $p_1$ divides $p_2$, a contradiction. So, $X^{p_2}$ must be cancelled out by $X^{mp_1+1}$ for some $m$. Therefore, $p_2 = mp_1+1$, as desired. 
\end{proof}

\begin{proof}[Proof of Theorem \ref{S-}]
If $t$ is odd, Lemma \ref{key} says that $p_2$ and $p_2+1$ are in $S_n^-$. Suppose that $S_n^-$ is in arithmetic progression. Then $S_n^- = \{p_2, p_2+1, \ldots, \phi(n) - p_2\}$. Thus, the number of powers with negative coefficients is exactly $\phi(n) -2p_2 + 1$. Hence, $-(\phi(n)-2p_2+1)$ is an upper bound for the sum of these coefficients. By symmetry of cyclotomic polynomials, we have $$S_n^+ = \{0, 1, 2, \ldots, p_1-1,  \phi(n)-p_1+1,\ldots, \phi(n)-1, \phi(n)\}.$$ Thus, the number of powers with positive coefficient is exactly $2p_1$. Since each coefficient is 1, the sum of them is $2p_1$. Using the fact that $\Phi_n(1) = 1$ if $n$ is not a prime power, we know that $2p_1 - (\phi(n)-2p_2 + 1)\ \ge\ 1$,
which is equivalent to $2p_1 + 2p_2 \ge \phi(n)+2$. By Lemma \ref{skey}, we have $t = 2$, which contradicts the assumption that $t$ is odd. 

If $t$ is even, Lemma \ref{key} says that $1$ and $p_1+1$ are in $S_n^-$. If $p_2 + 1$ is in $S_n^-$, $p_1$ must divide $p_2$, a contradiction. So $X^{p_2+1}$ must be cancelled by $X^{mp_1}$ for some $m\in \mathbb{N}$. Therefore, $p_2 = mp_1 - 1$. 
\end{proof}
\section{Proof of Theorem \ref{i/e}.}
We first prove that $P_{a,b}(X)$ is a polynomial and then consider powers of monomials with positive coefficients. 
\begin{lem}\label{lemi/e}
For $1<a<b$ and $\gcd(a,b) = 1$, there exists a unique $3\le m\le b$ such that $b$ divides $(m-1)a-1$. 
\end{lem}
\begin{proof}
Because $(a,b) = 1$, there exist $r$ and $s$ such that $sa+tb = 1$. All integral solutions of the equation $xa+yb = 1$ are of the form $(x,y) = (r+tb, s-ta)$ for some $t\in \mathbb{Z}$. Hence, there is a unique solution with $1\le x = r_0+t_0b\le b$. Set $m = r_0+t_0b+1$. By definition, $b$ divides $(m-1)a-1$. It remains to show $3\le m\le b$ or equivalently, $1<r_0+t_0b<b$. If $r_0+t_0b = 1$; then $b$ divides $a-1$, which contradicts $1<a<b$. So $r_0+t_0b>1$. If $r_0+t_0b = b$, then $b$ divides $1$, which contradicts $b>1$. So $r_0+t_0b<b$. This completes the proof. 
\end{proof}

\begin{proof}[Proof of Theorem \ref{i/e}]
We write
\begin{align*}
    P_{a,b}(X)\ =\ \frac{(X-1)(X^{a(b-1)}+X^{a(b-2)}+\cdots + 1)}{X^b-1}. 
\end{align*}
It suffices to prove that $f(X): = (X-1)(X^{a(b-1)}+X^{a(b-2)}+\cdots+1)$ can be written as $(X^b-1)g(X)$ for some polynomial $f(X)$. We have
\begin{align}\label{ie1}f(X) \ =\ (X^{ab-a+1}+X^{ab-2a+1}+\cdots + X)-(X^{ab-a}+X^{ab-2a}+\cdots+1).\end{align}
Let $3\le m\le b$ be chosen such that $b$ divides $(m-1)a-1$. By Lemma \ref{lemi/e}, $m$ exists and is unique. For each $1\le k\le m-1$, we have
\begin{align}\label{ie2} X^{ab-ka+1}-X^{ab-(b+k-m+1)a}\ =\ X^{a(m-k-1)}(X^{ab-((m-1)a-1)}-1),\end{align}
which is divisible by $X^b-1$. 
For each $m\le k\le b$, we have 
\begin{align}\label{ie3}
X^{ab-ka+1}-X^{ab-(k-m+1)a}\ =\ X^{ab-ka+1}(1-X^{(m-1)a-1}),
\end{align}
which is divisible by $X^b-1$. From \eqref{ie1}, \eqref{ie2}, and \eqref{ie3}, we know that $P_{a,b}(X)$ is a polynomial. Furthermore, letting $\ell : = ((m-1)a-1)/b \ge 1$. we can write
\begin{align}\label{ie4}
P_{a,b}(X)\ =\ \sum_{k=1}^{m-1} X^{a(m-k-1)}u(X)-\sum_{k=m}^{b} X^{ab-ka+1}v(X),
\end{align}
where $u(X) = X^{(a-\ell-1)b}+X^{(a-\ell-2)b}+\cdots+1$ and $v(X) = X^{(\ell-1)b}+X^{(\ell-2)b}+\cdots+1$.

Next, we prove that exponents of monomials with positive coefficients are in arithmetic progression if and only if $b\equiv 1\Mod a$. Forward implication: By \eqref{ie4}, the two largest powers with positive coefficients are $a(m-2)+(a-\ell-1)b$ and $a(m-3)+(a-\ell-1)b$. (Note that the two monomials having these powers are not cancelled.) Hence, we have an arithmetic progression of difference $a$. If $u(X)$ has exactly one summand or $a-\ell-1 = 0$, then $b\equiv 1\Mod a$. Assume that $a-\ell-1 > 0$. Because $\gcd(a,b) = 1$, it follows that $X^{a(m-2)+(a-\ell-2)b}$ must get cancelled. Then there exist $1\le j\le \ell$ and $m\le k\le b$ such that
\begin{align*}
    a(m-2)+(a-\ell-2)b \ =\ (ab-ka+1)+b(\ell-j).
\end{align*}
Replacing $b\ell = (m-1)a-1$ and simplifying, we arrive at $a(k-m)+1 = b(2-j)$, which gives $j = 1$. So $a(k-m) + 1 = b$ and thus, $b\equiv 1\Mod a$. Backward implication: straightforward calculations show $m = b-(b-1)/a+1$ and $\ell = a-1$. Hence, $u(X) = 1$ and we have
\begin{align*}
    P_{a,b}(X) \ =\ \sum_{k=1}^{m-1} X^{a(m-k-1)}-\sum_{k=m}^{b} X^{ab-ka+1}(X^{(\ell-1)b}+X^{(\ell-2)b}+\cdots+1).
\end{align*}
Since the power of each monomial in the second sum is $(1-2b)\Mod a$, no summand in the first sum gets cancelled. Therefore, the powers of monomials with positive coefficient are in an arithmetic progression. We have shown that exponents of monomials with positive coefficients are in arithmetic progression if and only if $b\equiv 1\Mod a$. As the proof for negative coefficients is similar, we omit it.
\end{proof}

\ \\

\begin{thebibliography}{2}
\bibitem{Ba}
Bachman, G. (1993). \textit{On the coefficients of cyclotomic polynomials.} American Mathematical Society. 
\bibitem{Ba1}
Bachman, G. (2010). On ternary inclusion-exclusion polynomials. \textit{Integers} 10: 623--638.
\bibitem{Bm}
Bateman, P. (1949). Note on the coefficients of the cyclotomic polynomial. \textit{Bull. Amer. Math. Soc.} 55: 1180--1181.
\bibitem{Be1}
Beiter, M. (1964). The midterm coefficient of the cyclotomic polynomial $\Phi_{pq}(X)$. \textit{Amer. Math. Monthly} 71: 769--770.
\bibitem{Be2}
Beiter, M. (1968). Magnitude of the coefficients of the cyclotomic polynomials $\Phi_{pqr}(X)$. \textit{Amer. Math. Monthly} 75: 370--372.
\bibitem{Dr}
Dresden, G. (2004). On the middle coefficient of a cyclotomic polynomial. \textit{Amer. Math. Monthly} 111: 531--533.
\bibitem{LL}
Lam, T., Leung, K. (1996). On the cyclotomic polynomial $\Phi_{pq}(X)$. \textit{Amer. Math. Monthly} 103: 562--564.
\bibitem{Leh}
Lehmer, E. (1936). On the magnitude of the coefficients of the cyclotomic polynomials. \textit{Bull. Amer. Math. Soc.} 42: 389--392.
\bibitem{Mor}
Moree, P. (2014). Numerical semigroups, cyclotomic polynomials, and Bernoulli numbers. \textit{Amer. Math. Monthly} 121: 890--902.
\bibitem{Th}
Thangadurai, R. (2000). On the coefficients of cyclotomic polynomials. In: Adhikari, S., Katre, S., Thakur, D., eds. \textit{Cyclotomic Fields And Related Topics.} Bhaskaracharya Pratishthana, Pune, pp. 311--322.
\end{thebibliography}
\end{document}